\documentclass[11pt]{article}

\usepackage[letterpaper,margin=1.1in]{geometry}
\usepackage[T1]{fontenc}
\usepackage{amsmath,amsthm}
\usepackage{newpxtext}
\usepackage{newpxmath}
\usepackage{mathtools}
\usepackage{microtype}
\usepackage[hidelinks]{hyperref}
\usepackage[nameinlink,noabbrev]{cleveref}

\numberwithin{equation}{section}
\allowdisplaybreaks

\newtheorem{theorem}{Theorem}[section]
\newtheorem{proposition}[theorem]{Proposition}
\newtheorem{lemma}[theorem]{Lemma}
\newtheorem{corollary}[theorem]{Corollary}
\theoremstyle{definition}
\newtheorem{definition}[theorem]{Definition}
\newtheorem{example}[theorem]{Example}
\theoremstyle{remark}
\newtheorem{remark}[theorem]{Remark}

\crefname{theorem}{theorem}{theorems}
\crefname{proposition}{proposition}{propositions}
\crefname{lemma}{lemma}{lemmas}
\crefname{corollary}{corollary}{corollaries}
\crefname{definition}{definition}{definitions}
\crefname{example}{example}{examples}
\crefname{remark}{remark}{remarks}

\newcommand{\SP}{\mathrm{SP}}
\newcommand{\OP}{\mathrm{OP}}
\newcommand{\ellp}{\ell}
\newcommand{\vac}{\mathbf{1}}
\newcommand{\cQ}{\mathcal{Q}}
\newcommand{\Yop}{\mathcal{Y}}
\newcommand{\Ht}{\mathsf{H}}

\newcommand{\Pf}{\operatorname{Pf}}
\newcommand{\Pfint}{\operatorname{Pf}_{\mathrm{int}}}

\title{A Modified Greaves--Jing--Zhu Operator and a Shifted
\texorpdfstring{$t$}{t}-Gessel Formula}
\author{S.-J. Lee}
\date{}

\begin{document}

\maketitle

\begin{abstract}
The recent work of Greaves, Jing, and Zhu gives an operator construction
for the $t$-Schur functions and the $t$-Schur measure. 
Motivated by their construction, we consider the same type of vertex
operator on the odd power-sum ring. Its Fourier modes generate a family
of symmetric functions indexed by strict partitions, which we call
shifted $t$-Schur functions.
These functions specialize to Schur $Q$-functions at
$t=0$. We derive a
two-row formula, a Pfaffian Giambelli formula, a Cauchy identity, and a
finite shifted Gessel-type formula. This note is intended as a first step
toward further study of the odd-operator analogue of the Greaves--Jing--Zhu
construction.
\end{abstract}

\section{Introduction}

Vertex operators are useful tools in the theory of symmetric functions.
They are built from two simple operations: multiplication by power sums
and differentiation with respect to power sums. Taking exponentials of
these operations gives formal Laurent series whose coefficients act on
rings of symmetric functions. In this way, many classical symmetric
functions can be constructed by operators. For example, Bernstein
operators produce Schur functions, and related Jing vertex operators produce Hall--Littlewood functions. For general background,
see \cite{JingHL,Macdonald}.

There is also an odd version of this construction. Instead of using all
power sums, one works with the subring generated by
\[
p_1,p_3,p_5,\ldots .
\]
This odd power-sum ring is the natural setting for Schur $Q$-functions.
These functions are indexed by strict partitions and are closely related
to shifted tableaux and projective representations of symmetric groups;
see \cite{HoffmanHumphreys,Jing1991,Stembridge}. Thus Schur
$Q$-functions may be viewed as the shifted analogue of ordinary Schur
functions.

Schur-type measures give another reason for studying these functions.
Okounkov's Schur measure is built from products
$s_\lambda(X)s_\lambda(Y)$ \cite{Okounkov}. The shifted Schur measure of
Tracy and Widom is built from products
$Q_\lambda(X)P_\lambda(Y)$ and is supported on strict partitions
\cite{TracyWidom}. When one restricts the largest part of the partition,
one obtains Gessel-type determinant or Pfaffian formulas
\cite{Gessel,TracyWidom}.

Matsumoto introduced the $t$-Schur measure and studied its asymptotic
behavior \cite{MatsumotoT}. More recently, Greaves, Jing, and Zhu
\cite{GJZ} gave a vertex-operator construction for the $t$-Schur
functions and the corresponding measure. One of their operators can be
written as
\begin{equation}\label{eq:GJZ-intro}
Y(z;t)
=
\exp\left(
\sum_{n\geq 1}\frac{1-t^n}{n}p_nz^n
\right)
\exp\left(
-\sum_{n\geq 1}\frac{1}{1-t^n}
\frac{\partial}{\partial p_n}z^{-n}
\right).
\end{equation}
The two factors involving $1-t^n$ and $(1-t^n)^{-1}$ cancel in the
basic commutator. Because of this cancellation, the operator product has
the same simple form as in the ordinary Schur case.

In this note, we apply the same idea to the odd power-sum ring. Namely,
we consider the operator
\begin{equation}\label{eq:odd-GJZ-intro}
\Yop(z;t)
=
\exp\left(
\sum_{\substack{n\geq 1\\ n\,\mathrm{odd}}}
\frac{2(1-t^n)}{n}p_nz^n
\right)
\exp\left(
-\sum_{\substack{n\geq 1\\ n\,\mathrm{odd}}}
\frac{1}{1-t^n}
\frac{\partial}{\partial p_n}z^{-n}
\right).
\end{equation}
The factor $2$ is the usual normalization in the Schur $Q$-function
case. A direct calculation gives
\[
\Yop(z;t)\Yop(w;t)
=
:\!\Yop(z;t)\Yop(w;t)\!:
\frac{1-w/z}{1+w/z}.
\]
Hence the parameter $t$ does not change the basic anti-commutation
relation of the modes. If
\[
\Yop(z;t)=\sum_{m\in\mathbb Z}\Yop_m(t)z^{-m},
\]
then
\[
\Yop_m(t)\Yop_n(t)+\Yop_n(t)\Yop_m(t)
=
2(-1)^m\delta_{m,-n}.
\]
This is the same relation that appears in the usual operator
construction of Schur $Q$-functions.

For a strict partition
$\lambda=(\lambda_1>\cdots>\lambda_{\ellp}>0)$, we define
\begin{equation}\label{eq:shifted-intro-def}
\cQ_\lambda(X;t)
=
\Yop_{-\lambda_1}(t)\cdots
\Yop_{-\lambda_{\ellp}}(t)\vac.
\end{equation}
For convenience, we call these functions shifted $t$-Schur functions.
When $t=0$, the operator $\Yop(z;t)$ becomes the standard Schur
$Q$-function vertex operator, and therefore
\[
\cQ_\lambda(X;0)=Q_\lambda(X).
\]

The first identity proved in this paper is the Cauchy identity
\begin{equation}\label{eq:Cauchy-intro}
\sum_{\lambda\in\SP}
\cQ_\lambda(X;t)P_\lambda(Y)
=
\prod_{i,j\geq 1}
\frac{(1+x_i y_j)(1-tx_i y_j)}
     {(1-x_i y_j)(1+tx_i y_j)},
\end{equation}
where $\SP$ is the set of strict partitions and
$P_\lambda=2^{-\ellp(\lambda)}Q_\lambda$. This reduces to the usual
Schur $Q$-Cauchy identity when $t=0$.

We also prove a finite shifted Gessel-type formula. For a positive
integer $h$, we define two finite skew-symmetric matrices
$A_h^{(t)}(X)$ and $B_h(Y)$ from the one-row and two-row functions. Then
\begin{equation}\label{eq:Gessel-intro}
\sum_{\substack{\lambda\in\SP\\ \lambda_1\leq h}}
\cQ_\lambda(X;t)P_\lambda(Y)
=
\Pfint
\begin{pmatrix}
-A_h^{(t)}(X)&I_{h+1}\\
-I_{h+1}&B_h(Y)
\end{pmatrix}
=
\det\left(
I_{h+1}-A_h^{(t)}(X)B_h(Y)
\right)^{1/2}.
\end{equation}
Here the first Pfaffian is taken in an interlaced order, which is
explained in \cref{sec:Gessel}. At $t=0$, this formula becomes the
shifted Gessel formula of Tracy and Widom \cite{TracyWidom}.

The paper is organized as follows. In \cref{sec:odd-operator}, we
define the odd Greaves--Jing--Zhu operator and derive its
anti-commutation relation, its rescaling property, and its Pfaffian
formula. In \cref{sec:Cauchy}, we prove the Cauchy identity. In
\cref{sec:Gessel}, we use a finite Pfaffian summation rule to prove the
truncated shifted Gessel formula.

\section{The Greaves-Jing-Zhu operator}\label{sec:odd-operator}

\subsection{The odd Heisenberg algebra}

Let
\[
\Gamma=\mathbb{Q}(t)[p_1,p_3,p_5,\ldots]
\]
be the ring generated by the odd power-sum symmetric functions. We
regard $\Gamma$ as a bosonic Fock space with vacuum vector $\vac=1$.
For every positive odd integer $n$, define
\begin{equation}\label{eq:odd-Heis}
a_{-n}=p_n,
\qquad
a_n=\frac{n}{2}\frac{\partial}{\partial p_n}.
\end{equation}
Then
\begin{equation}\label{eq:odd-Heis-rel}
[a_m,a_n]=\frac{m}{2}\delta_{m,-n}
\end{equation}
for nonzero odd integers $m,n$, and $a_n\vac=0$ for $n>0$.

\begin{definition}\label{def:odd-GJZ}
The \emph{odd GJZ operator} is
\begin{align}
\Yop(z;t)
&=
\exp\left(
\sum_{\substack{n\geq 1\\ n\,\mathrm{odd}}}
\frac{2(1-t^n)}{n}a_{-n}z^n
\right)
\exp\left(
-\sum_{\substack{n\geq 1\\ n\,\mathrm{odd}}}
\frac{2}{n(1-t^n)}a_nz^{-n}
\right)\notag\\
&=
\exp\left(
\sum_{\substack{n\geq 1\\ n\,\mathrm{odd}}}
\frac{2(1-t^n)}{n}p_nz^n
\right)
\exp\left(
-\sum_{\substack{n\geq 1\\ n\,\mathrm{odd}}}
\frac{1}{1-t^n}
\frac{\partial}{\partial p_n}z^{-n}
\right)\notag\\
&=\sum_{m\in\mathbb Z}\Yop_m(t)z^{-m}.
\label{eq:odd-GJZ}
\end{align}
\end{definition}

The normal-ordered product, denoted by $:\!\cdots\!:$, is obtained by
moving all creation operators $a_{-n}$ to the left of all annihilation
operators $a_n$.

\begin{proposition}[Operator product expansion]\label{prop:OPE}
The odd GJZ operator satisfies
\begin{equation}\label{eq:OPE}
\Yop(z;t)\Yop(w;t)
=:\!\Yop(z;t)\Yop(w;t)\!:
\frac{1-w/z}{1+w/z},
\end{equation}
where the rational factor is expanded as a formal power series in
$w/z$.
\end{proposition}

\begin{proof}
Write $\Yop(z;t)=e^{A(z)}e^{B(z)}$, where
\[
A(z)=\sum_{\substack{n\geq 1\\n\,\mathrm{odd}}}
\frac{2(1-t^n)}{n}a_{-n}z^n,
\qquad
B(z)=-\sum_{\substack{n\geq 1\\n\,\mathrm{odd}}}
\frac{2}{n(1-t^n)}a_nz^{-n}.
\]
By \eqref{eq:odd-Heis-rel},
\begin{align*}
[B(z),A(w)]
&=-2\sum_{\substack{n\geq 1\\n\,\mathrm{odd}}}
\frac{1}{n}\left(\frac{w}{z}\right)^n\\
&=\log\frac{1-w/z}{1+w/z}.
\end{align*}
The commutator is scalar, so the Baker--Campbell--Hausdorff formula
gives \eqref{eq:OPE}.
\end{proof}

Let
\[
\delta(u)=\sum_{k\in\mathbb Z}u^k
\]
be the formal delta function. Since only odd Heisenberg modes occur in
\eqref{eq:odd-GJZ},
\begin{equation}\label{eq:normal-at-minus}
:\!\Yop(z;t)\Yop(-z;t)\!:=1.
\end{equation}
We also use the formal identity
\begin{equation}\label{eq:delta-rational}
\iota_{z,w}\frac{1-w/z}{1+w/z}
+
\iota_{w,z}\frac{1-z/w}{1+z/w}
=2\delta(-w/z),
\end{equation}
where $\iota_{z,w}$ and $\iota_{w,z}$ denote the two formal expansions.

\begin{corollary}\label{cor:Clifford}
The operator $\Yop(z;t)$ satisfies
\begin{equation}\label{eq:field-Clifford}
\Yop(z;t)\Yop(w;t)+\Yop(w;t)\Yop(z;t)
=2\delta(-w/z).
\end{equation}
Equivalently,
\begin{equation}\label{eq:mode-Clifford}
\Yop_m(t)\Yop_n(t)+\Yop_n(t)\Yop_m(t)
=2(-1)^m\delta_{m,-n}.
\end{equation}
\end{corollary}

\begin{proof}
Adding the two expansions in
\cref{prop:OPE} and using \eqref{eq:delta-rational} gives
\[
2:\!\Yop(z;t)\Yop(w;t)\!:\,\delta(-w/z).
\]
For every formal series $F(z,w)$,
$F(z,w)\delta(-w/z)=F(z,-z)\delta(-w/z)$. Hence
\eqref{eq:normal-at-minus} gives \eqref{eq:field-Clifford}. Comparing
the coefficient of $z^{-m}w^{-n}$ gives \eqref{eq:mode-Clifford}.
\end{proof}

In particular, if $r,s>0$, then
\begin{equation}\label{eq:negative-anticommute}
\Yop_{-r}(t)\Yop_{-s}(t)
=-\Yop_{-s}(t)\Yop_{-r}(t),
\qquad
\Yop_{-r}(t)^2=0.
\end{equation}
This explains the use of strict partitions.

\subsection{The functions and their Pfaffian formula}

Acting on the vacuum gives the one-row series
\begin{align}
\Ht_t(X;z)
&:=\Yop(z;t)\vac\notag\\
&=\exp\left(
2\sum_{\substack{n\geq 1\\n\,\mathrm{odd}}}
\frac{1-t^n}{n}p_n(X)z^n
\right)
=\sum_{r\geq 0}q_r(X;t)z^r.
\label{eq:Ht-def}
\end{align}
When $t=0$, we write $q_r(X)=q_r(X;0)$. The logarithmic expansion gives
\begin{equation}\label{eq:Ht-product}
\Ht_t(X;z)
=
\prod_{i\geq 1}
\frac{(1+x_i z)(1-tx_i z)}
     {(1-x_i z)(1+tx_i z)}.
\end{equation}

A strict partition is a finite sequence
\[
\lambda=(\lambda_1>\lambda_2>\cdots>\lambda_{\ellp}>0).
\]
We write $\SP$ for the set of strict partitions, including the empty
partition. When a condition involves the largest part of the empty
partition, we use the convention $\lambda_1=0$.

\begin{definition}\label{def:shifted-t-Schur}
For $\lambda\in\SP$, define
\begin{equation}\label{eq:shifted-t-Schur-def}
\cQ_\lambda(X;t)
=
\Yop_{-\lambda_1}(t)\Yop_{-\lambda_2}(t)
\cdots\Yop_{-\lambda_{\ellp}}(t)\vac.
\end{equation}
We call this the \emph{shifted $t$-Schur function} attached to $\lambda$,
and we set $\cQ_{\varnothing}(X;t)=1$.
\end{definition}

Let
\begin{equation}\label{eq:standard-neutral-field}
\Phi(z):=\Yop(z;0)=\sum_{m\in\mathbb Z}\phi_mz^{-m}.
\end{equation}
The usual operator construction gives
\begin{equation}\label{eq:Q-state-early}
Q_\lambda
=
\phi_{-\lambda_1}\cdots\phi_{-\lambda_{\ellp}}\vac.
\end{equation}

\begin{remark}\label{rem:t1}
The operator $\Yop(z;t)$ is written over $\mathbb Q(t)$ and should not
be specialized directly at $t=1$. 
\end{remark}

\begin{proposition}\label{prop:multipoint}
For $k\geq 1$,
\begin{equation}\label{eq:multipoint}
\Yop(z_1;t)\cdots\Yop(z_k;t)\vac
=
\prod_{1\leq i<j\leq k}
\frac{z_i-z_j}{z_i+z_j}
\prod_{i=1}^k\Ht_t(X;z_i),
\end{equation}
where every factor is expanded in the formal region
$|z_1|>|z_2|>\cdots>|z_k|$. Consequently,
\begin{equation}\label{eq:coefficient-shifted}
\cQ_\lambda(X;t)
=
[z_1^{\lambda_1}\cdots z_k^{\lambda_k}]
\prod_{1\leq i<j\leq k}
\frac{z_i-z_j}{z_i+z_j}
\prod_{i=1}^k\Ht_t(X;z_i).
\end{equation}
\end{proposition}

\begin{proof}
Repeated use of \eqref{eq:OPE} gives the rational factor in
\eqref{eq:multipoint}. The annihilation part of the normal-ordered
product acts trivially on $\vac$, while the creation part gives the
product of one-row series. Coefficient extraction proves
\eqref{eq:coefficient-shifted}.
\end{proof}

For $r>s\geq 0$, define the  two-row function
\begin{equation}\label{eq:two-row-def}
\cQ_{(r,s)}(X;t)
=
\Yop_{-r}(t)\Yop_{-s}(t)\vac.
\end{equation}
We extend this notation by skew-symmetry and set
$\cQ_{(0,0)}(X;t)=0$. In particular,
\begin{equation}\label{eq:r0}
\cQ_{(r,0)}(X;t)=q_r(X;t).
\end{equation}
Since
\[
\frac{1-u}{1+u}=1+2\sum_{k\geq 1}(-1)^ku^k,
\]
\cref{prop:multipoint} gives
\begin{equation}\label{eq:two-row-formula}
\cQ_{(r,s)}(X;t)
=
q_r(X;t)q_s(X;t)
+2\sum_{k=1}^{s}(-1)^k
q_{r+k}(X;t)q_{s-k}(X;t)
\end{equation}
for $r>s\geq 0$.

\begin{theorem}
\label{thm:Pfaffian-Giambelli}
Let $\lambda=(\lambda_1>\cdots>\lambda_{\ellp}>0)$ be a strict
partition. If $\ellp$ is odd, append a zero part. If $2m$ is the
resulting even length, then
\begin{equation}\label{eq:Pfaffian-Giambelli}
\cQ_\lambda(X;t)
=
\Pf\left(
\cQ_{(\lambda_i,\lambda_j)}(X;t)
\right)_{1\leq i,j\leq 2m}.
\end{equation}
\end{theorem}

\begin{proof}
Assume first that $\ellp=2m$. Schur's Pfaffian identity is
\begin{equation}\label{eq:Schur-Pf}
\prod_{1\leq i<j\leq 2m}
\frac{z_i-z_j}{z_i+z_j}
=
\Pf\left(
\frac{z_i-z_j}{z_i+z_j}
\right)_{1\leq i,j\leq 2m}.
\end{equation}
Multiplying the $i$th row and column by $\Ht_t(X;z_i)$ gives
\begin{align*}
&\prod_{1\leq i<j\leq 2m}
\frac{z_i-z_j}{z_i+z_j}
\prod_{i=1}^{2m}\Ht_t(X;z_i)\\
&\qquad=
\Pf\left(
\frac{z_i-z_j}{z_i+z_j}
\Ht_t(X;z_i)\Ht_t(X;z_j)
\right)_{1\leq i,j\leq 2m}.
\end{align*}
Extracting the coefficient of
$z_1^{\lambda_1}\cdots z_{2m}^{\lambda_{2m}}$ and using
multilinearity gives \eqref{eq:Pfaffian-Giambelli}. If $\ellp$ is odd,
then $\Yop_0(t)\vac=\vac$, so we append a zero part and apply the even
case.
\end{proof}

\section{The shifted \texorpdfstring{$t$}{t}-Cauchy identity}
\label{sec:Cauchy}

An odd partition is a partition all of whose parts are odd. We write
$\OP$ for the set of odd partitions. For $\alpha\in\OP$, let
$m_r(\alpha)$ be the multiplicity of the part $r$ in $\alpha$, and set
\[
z_\alpha
=
\prod_{r\geq 1}r^{m_r(\alpha)}m_r(\alpha)!.
\]
The standard scalar product on $\Gamma$ is
\begin{equation}\label{eq:Q-inner-product}
\langle p_\alpha,p_\beta\rangle
=
2^{-\ellp(\alpha)}z_\alpha\delta_{\alpha\beta},
\qquad
\alpha,\beta\in\OP.
\end{equation}
With respect to this scalar product, $a_n^*=a_{-n}$.

For the field in \eqref{eq:standard-neutral-field},
\cref{cor:Clifford} gives
\begin{equation}\label{eq:standard-Clifford}
\phi_m\phi_n+\phi_n\phi_m
=
2(-1)^m\delta_{m,-n}.
\end{equation}
Moreover,
\begin{equation}\label{eq:adjoint-neutral}
\Phi(z)^*=\Phi(-z^{-1}),
\qquad
\phi_m^*=(-1)^m\phi_{-m}.
\end{equation}
The Clifford relations imply
\begin{equation}\label{eq:Q-orthogonal}
\langle Q_\lambda,Q_\mu\rangle
=
2^{\ellp(\lambda)}\delta_{\lambda\mu}.
\end{equation}
We write
\begin{equation}\label{eq:P-def}
P_\lambda=2^{-\ellp(\lambda)}Q_\lambda.
\end{equation}
Thus the $P$- and $Q$-bases are dual.

For an alphabet $X$, define the half-vertex operator
\begin{equation}\label{eq:Gamma-plus-t}
\Gamma_{+,t}(X)
=
\exp\left(
\sum_{\substack{n\geq 1\\n\,\mathrm{odd}}}
\frac{2(1-t^n)}{n}p_n(X)a_n
\right).
\end{equation}
Since $[a_n,\Phi(z)]=z^n\Phi(z)$, we have
\begin{equation}\label{eq:Gamma-conjugation}
\Gamma_{+,t}(X)\Phi(z)
=
\Ht_t(X;z)\Phi(z)\Gamma_{+,t}(X).
\end{equation}

\begin{proposition}
\label{prop:matrix-element}
For every strict partition $\lambda$,
\begin{equation}\label{eq:matrix-element}
\cQ_\lambda(X;t)
=
\left\langle
\vac,
\Gamma_{+,t}(X)
\phi_{-\lambda_1}\cdots\phi_{-\lambda_{\ellp}}
\vac
\right\rangle.
\end{equation}
\end{proposition}

\begin{proof}
Using \eqref{eq:Gamma-conjugation}, the ordinary neutral OPE, and the
vacuum conditions, we obtain
\begin{align*}
&\left\langle
\vac,
\Gamma_{+,t}(X)
\Phi(z_1)\cdots\Phi(z_k)
\vac
\right\rangle\\
&\qquad=
\prod_{i=1}^k\Ht_t(X;z_i)
\left\langle
\vac,
\Phi(z_1)\cdots\Phi(z_k)
\vac
\right\rangle\\
&\qquad=
\prod_{1\leq i<j\leq k}
\frac{z_i-z_j}{z_i+z_j}
\prod_{i=1}^k\Ht_t(X;z_i).
\end{align*}
This is the same generating function as \eqref{eq:multipoint}.
Comparing coefficients gives \eqref{eq:matrix-element}.
\end{proof}

Define
\begin{equation}\label{eq:Gamma-minus}
\Gamma_-(Y)
=
\exp\left(
\sum_{\substack{n\geq 1\\n\,\mathrm{odd}}}
\frac{2}{n}p_n(Y)a_{-n}
\right).
\end{equation}

\begin{lemma}\label{lem:reproducing}
For every $f\in\Gamma$,
\begin{equation}\label{eq:reproducing}
\left\langle f,\Gamma_-(Y)\vac\right\rangle=f(Y).
\end{equation}
Consequently,
\begin{equation}\label{eq:Gamma-minus-expansion}
\Gamma_-(Y)\vac
=
\sum_{\lambda\in\SP}P_\lambda(Y)Q_\lambda.
\end{equation}
\end{lemma}

\begin{theorem}\label{thm:shifted-Cauchy}
For two alphabets $X=(x_1,x_2,\ldots)$ and
$Y=(y_1,y_2,\ldots)$,
\begin{align}
\sum_{\lambda\in\SP}
\cQ_\lambda(X;t)P_\lambda(Y)
&=
\exp\left(
2\sum_{\substack{n\geq 1\\n\,\mathrm{odd}}}
\frac{1-t^n}{n}p_n(X)p_n(Y)
\right)
\label{eq:shifted-Cauchy-exp}\\
&=
\prod_{i,j\geq 1}
\frac{(1+x_i y_j)(1-tx_i y_j)}
     {(1-x_i y_j)(1+tx_i y_j)}.
\label{eq:shifted-Cauchy-product}
\end{align}
\end{theorem}

\begin{proof}
Consider
\begin{equation}\label{eq:Cauchy-correlator}
Z_t(X,Y)
=
\left\langle
\vac,
\Gamma_{+,t}(X)\Gamma_-(Y)\vac
\right\rangle.
\end{equation}
Let
\[
U
=
\sum_{\substack{n\geq 1\\n\,\mathrm{odd}}}
\frac{2(1-t^n)}{n}p_n(X)a_n,
\qquad
V
=
\sum_{\substack{n\geq 1\\n\,\mathrm{odd}}}
\frac{2}{n}p_n(Y)a_{-n}.
\]
By \eqref{eq:odd-Heis-rel},
\begin{equation}\label{eq:Cauchy-commutator}
[U,V]
=
2\sum_{\substack{n\geq 1\\n\,\mathrm{odd}}}
\frac{1-t^n}{n}p_n(X)p_n(Y).
\end{equation}
This is scalar. The Baker--Campbell--Hausdorff formula and the vacuum
conditions therefore give
\begin{equation}\label{eq:Cauchy-first-evaluation}
Z_t(X,Y)
=
\exp\left(
2\sum_{\substack{n\geq 1\\n\,\mathrm{odd}}}
\frac{1-t^n}{n}p_n(X)p_n(Y)
\right).
\end{equation}
On the other hand, \cref{lem:reproducing,prop:matrix-element} give
\begin{align*}
Z_t(X,Y)
&=
\sum_{\lambda\in\SP}
P_\lambda(Y)
\left\langle
\vac,\Gamma_{+,t}(X)Q_\lambda
\right\rangle\\
&=
\sum_{\lambda\in\SP}
\cQ_\lambda(X;t)P_\lambda(Y).
\end{align*}
This proves \eqref{eq:shifted-Cauchy-exp}.

Finally, for $u=x_i y_j$,
\begin{align*}
2\sum_{\substack{n\geq 1\\n\,\mathrm{odd}}}
\frac{1-t^n}{n}u^n
&=
\log\frac{1+u}{1-u}
-
\log\frac{1+tu}{1-tu}\\
&=
\log\frac{(1+u)(1-tu)}{(1-u)(1+tu)}.
\end{align*}
Taking the product over $i,j$ proves
\eqref{eq:shifted-Cauchy-product}.
\end{proof}

\section{The shifted \texorpdfstring{$t$}{t}-Gessel formula}
\label{sec:Gessel}

We now truncate the sum by the largest part. The finite identity below
is a special case of the Pfaffian minor-summation formula of Ishikawa
and Wakayama \cite{IshikawaWakayama}. We include a short proof because
the interlaced order is important for the signs.

\subsection{A finite Pfaffian summation rule}

Let $I=(i_1,\ldots,i_N)$ be a finite ordered set, and let
$A=(a_{rs})_{r,s\in I}$ and $B=(b_{rs})_{r,s\in I}$ be skew-symmetric
matrices. For an even subset $J\subseteq I$, let $A[J]$ and $B[J]$
denote the principal submatrices in the induced order. We use the
convention $\Pf(A[\varnothing])=1$.

The notation
\[
\Pfint
\begin{pmatrix}
-A&I_N\\
-I_N&B
\end{pmatrix}
\]
means that the rows and columns are read in the order
\[
u_{i_1},v_{i_1},u_{i_2},v_{i_2},\ldots,u_{i_N},v_{i_N},
\]
even though the displayed block matrix uses the order
\[
u_{i_1},\ldots,u_{i_N},v_{i_1},\ldots,v_{i_N}.
\]

\begin{lemma}\label{lem:finite-Wick}
With the notation above,
\begin{equation}\label{eq:finite-Wick}
\Pfint
\begin{pmatrix}
-A&I_N\\
-I_N&B
\end{pmatrix}
=
\sum_{\substack{J\subseteq I\\ |J|\,\mathrm{even}}}
\Pf(A[J])\Pf(B[J]).
\end{equation}
\end{lemma}

\begin{proof}
Introduce Grassmann variables $\xi_r,\eta_r$ in the interlaced order
$\xi_{i_1},\eta_{i_1},\ldots,\xi_{i_N},\eta_{i_N}$. The left-hand side
of \eqref{eq:finite-Wick} is the coefficient of
$\xi_{i_1}\eta_{i_1}\cdots\xi_{i_N}\eta_{i_N}$ in
\begin{equation}\label{eq:Grassmann-exp}
\exp\left(
-\frac12\xi^{T}A\xi
+\sum_{r\in I}\xi_r\eta_r
+\frac12\eta^{T}B\eta
\right).
\end{equation}
Expand the middle exponential as
\[
\prod_{r\in I}(1+\xi_r\eta_r).
\]
Choose the mixed factor for every $r\notin J$. The remaining variables
have indices in $J$, so $|J|$ must be even. If $|J|=2s$, the
$\xi$-quadratic part contributes
$\Pf(-A[J])=(-1)^s\Pf(A[J])$. Reordering the remaining $\xi$- and
$\eta$-variables into the interlaced order contributes the same sign
$(-1)^s$. The two signs cancel, and the $\eta$-quadratic part
contributes $\Pf(B[J])$. Summing over $J$ gives
\eqref{eq:finite-Wick}.
\end{proof}

\subsection{The truncated sum}

Fix $h\geq 1$ and use the ordered index set
\begin{equation}\label{eq:index-order}
I_h=(h,h-1,\ldots,1,0).
\end{equation}
Define the $(h+1)\times(h+1)$ skew-symmetric matrix
\begin{equation}\label{eq:A-matrix-def}
A_h^{(t)}(X)
=
\left(a_{rs}^{(t)}(X)\right)_{r,s\in I_h}
\end{equation}
by
\begin{equation}\label{eq:A-entries}
a_{rs}^{(t)}(X)
=
\begin{cases}
\cQ_{(r,s)}(X;t),&r>s\geq 1,\\
q_r(X;t),&r\geq 1,\ s=0,\\
-a_{sr}^{(t)}(X),&r<s,\\
0,&r=s.
\end{cases}
\end{equation}
By \cref{prop:matrix-element}, for $r\neq s$ these entries can also be
written as
\begin{equation}\label{eq:A-as-contraction}
a_{rs}^{(t)}(X)
=
\left\langle
\vac,
\Gamma_{+,t}(X)\phi_{-r}\phi_{-s}\vac
\right\rangle.
\end{equation}
Here $\phi_0\vac=\vac$, while the diagonal entries of
$A_h^{(t)}(X)$ are set equal to zero.

Let
\begin{equation}\label{eq:D-matrix}
D_h
=
\operatorname{diag}(d_h,d_{h-1},\ldots,d_1,d_0),
\qquad
d_r=\frac12\ (r\geq 1),
\quad
d_0=1,
\end{equation}
and set
\begin{equation}\label{eq:B-matrix-def}
B_h(Y)=D_hA_h^{(0)}(Y)D_h.
\end{equation}
Thus, for $r>s\geq 1$,
\begin{equation}\label{eq:B-entries}
(B_h(Y))_{rs}
=
\frac14Q_{(r,s)}(Y),
\qquad
(B_h(Y))_{r0}
=
\frac12q_r(Y).
\end{equation}

For a strict partition
$\lambda=(\lambda_1>\cdots>\lambda_{\ellp}>0)$ with
$\lambda_1\leq h$, define an even ordered subset of $I_h$ by
\begin{equation}\label{eq:J-lambda}
J_\lambda
=
\begin{cases}
(\lambda_1,\ldots,\lambda_{\ellp}),
&\ellp\ \text{even},\\
(\lambda_1,\ldots,\lambda_{\ellp},0),
&\ellp\ \text{odd}.
\end{cases}
\end{equation}
For the empty partition, set $J_{\varnothing}=\varnothing$. Since the
order in \eqref{eq:index-order} is decreasing,
\cref{thm:Pfaffian-Giambelli} gives
\begin{equation}\label{eq:principal-A}
\Pf\left(A_h^{(t)}(X)[J_\lambda]\right)
=
\cQ_\lambda(X;t).
\end{equation}
Also, $J_\lambda$ contains exactly $\ellp(\lambda)$ positive indices,
so
\begin{align}
\Pf\left(B_h(Y)[J_\lambda]\right)
&=
\det\left(D_h[J_\lambda]\right)
\Pf\left(A_h^{(0)}(Y)[J_\lambda]\right)\notag\\
&=
2^{-\ellp(\lambda)}Q_\lambda(Y)
=
P_\lambda(Y).
\label{eq:principal-B}
\end{align}
Every even subset of $I_h$ is uniquely of the form $J_\lambda$.

\begin{theorem}\label{thm:shifted-Gessel}
For every integer $h\geq 1$,
\begin{align}
\sum_{\substack{\lambda\in\SP\\\lambda_1\leq h}}
\cQ_\lambda(X;t)P_\lambda(Y)
&=
\Pfint
\begin{pmatrix}
-A_h^{(t)}(X)&I_{h+1}\\
-I_{h+1}&B_h(Y)
\end{pmatrix}
\label{eq:shifted-Gessel-Pf}\\
&=
\det\left(
I_{h+1}-A_h^{(t)}(X)B_h(Y)
\right)^{1/2}.
\label{eq:shifted-Gessel-det}
\end{align}
\end{theorem}

\begin{proof}
By \eqref{eq:principal-A}, \eqref{eq:principal-B}, and the bijection
between strict partitions and even subsets of $I_h$,
\begin{align*}
&\sum_{\substack{\lambda\in\SP\\\lambda_1\leq h}}
\cQ_\lambda(X;t)P_\lambda(Y)\\
&\qquad=
\sum_{\substack{J\subseteq I_h\\|J|\,\mathrm{even}}}
\Pf\left(A_h^{(t)}(X)[J]\right)
\Pf\left(B_h(Y)[J]\right).
\end{align*}
Now apply \cref{lem:finite-Wick} to obtain
\eqref{eq:shifted-Gessel-Pf}.

The square of a Pfaffian is the determinant of its skew-symmetric
matrix. Also,
\begin{equation}\label{eq:block-det}
\det
\begin{pmatrix}
-A&I\\
-I&B
\end{pmatrix}
=
\det(I-AB).
\end{equation}
For example, this follows first when $A$ is invertible by a Schur
complement, and then in general by polynomial continuation. The
Pfaffian in \eqref{eq:shifted-Gessel-Pf} has constant term $1$, which fixes
the square-root branch and proves \eqref{eq:shifted-Gessel-det}.
\end{proof}

It is useful to put the determinant in the symmetric normalization used
for the shifted Gessel formula. Define
\begin{equation}\label{eq:K-def}
K_h^{(t)}(X)
=
D_h^{1/2}A_h^{(t)}(X)D_h^{1/2},
\qquad
K_h(Y)
=
D_h^{1/2}A_h^{(0)}(Y)D_h^{1/2}.
\end{equation}
Explicitly,
\begin{equation}\label{eq:K-entries}
(K_h^{(t)}(X))_{rs}
=
\begin{cases}
\dfrac12\cQ_{(r,s)}(X;t),&r,s\geq 1,\\[4pt]
\dfrac{1}{\sqrt2}q_r(X;t),&r\geq 1,\ s=0,\\[4pt]
-\dfrac{1}{\sqrt2}q_s(X;t),&r=0,\ s\geq 1,\\[4pt]
0,&r=s=0.
\end{cases}
\end{equation}
The matrices $A_h^{(t)}(X)B_h(Y)$ and
$K_h^{(t)}(X)K_h(Y)$ are similar. Hence
\begin{equation}\label{eq:shifted-Gessel-K}
\sum_{\substack{\lambda\in\SP\\\lambda_1\leq h}}
\cQ_\lambda(X;t)P_\lambda(Y)
=
\det\left(
I_{h+1}-K_h^{(t)}(X)K_h(Y)
\right)^{1/2}.
\end{equation}

\begin{corollary}\label{cor:t0-Gessel}
At $t=0$, \eqref{eq:shifted-Gessel-K} becomes
\begin{equation}\label{eq:shifted-Gessel}
\sum_{\substack{\lambda\in\SP\\\lambda_1\leq h}}
Q_\lambda(X)P_\lambda(Y)
=
\det\left(
I_{h+1}-K_h(X)K_h(Y)
\right)^{1/2},
\end{equation}
which is the shifted Gessel formula.
\end{corollary}

\begin{example}\label{ex:h1}
For $h=1$, the strict partitions are $\varnothing$ and $(1)$. Hence
\[
\sum_{\lambda_1\leq 1}
\cQ_\lambda(X;t)P_\lambda(Y)
=
1+\frac12q_1(X;t)q_1(Y).
\]
In the order $(1,0)$,
\[
A_1^{(t)}(X)
=
\begin{pmatrix}
0&q_1(X;t)\\
-q_1(X;t)&0
\end{pmatrix},
\qquad
B_1(Y)
=
\begin{pmatrix}
0&\frac12q_1(Y)\\
-\frac12q_1(Y)&0
\end{pmatrix}.
\]
Therefore,
\[
\det\left(I_2-A_1^{(t)}(X)B_1(Y)\right)^{1/2}
=
1+\frac12q_1(X;t)q_1(Y),
\]
as expected.
\end{example}

Finally, the Cauchy identity is the coefficientwise stable limit of the
Gessel formula. For every fixed total degree, the condition
$\lambda_1\leq h$ is automatic when $h$ is sufficiently large. Thus
\cref{thm:shifted-Cauchy,thm:shifted-Gessel} give
\begin{equation}\label{eq:stable-limit}
\lim_{h\to\infty}
\det\left(
I_{h+1}-K_h^{(t)}(X)K_h(Y)
\right)^{1/2}
=
\prod_{i,j\geq 1}
\frac{(1+x_i y_j)(1-tx_i y_j)}
     {(1-x_i y_j)(1+tx_i y_j)}.
\end{equation}

\end{document}